\pdfoutput=1 
\documentclass[final]{cmslatex}

\usepackage{graphicx}
\usepackage[update]{epstopdf}
\usepackage{url}
\usepackage{hyperref}
\hypersetup{colorlinks=true,linkcolor=blue,citecolor=blue,urlcolor=blue}
\usepackage{graphicx}
\usepackage{epstopdf}
\pdfoutput=1

\usepackage{bm}
\usepackage{amsmath}
\usepackage{amssymb}
\usepackage{graphicx}
\usepackage{subcaption}
\usepackage{color}
\usepackage{amsfonts}
\usepackage{amscd}
\usepackage{mathrsfs}
\usepackage{enumerate}
\usepackage{url}
\providecommand{\email}[1]{\url{#1}}
\usepackage{algorithm}
\usepackage{algpseudocode}
\usepackage{amsmath, amssymb, bm, mathtools}
\usepackage{empheq}
\usepackage{tikz}
\usetikzlibrary{arrows.meta, decorations.pathreplacing}

\newcommand{\R}{\mathbb{R}}

\def\d{\ensuremath{\mathrm{d}}}

\newcommand{\norm}[1]{\left\|#1\right\|}

\newcommand{\nn}{{\mathbf{n}}}

\newenvironment{equationa*}{\begin{equation*}\begin{aligned}} {\end{aligned}\end{equation*}}
\allowdisplaybreaks[2]

\newcommand{\Th}{\mathcal{T}_h}
\newcommand{\Kh}{\mathcal{K}_h}

\newcommand{\Sh}{S_h}

\newcommand{\diam}{\operatorname{diam}}

\newcommand{\wh}{\widehat}
\newcommand{\wt}{\widetilde}
\newcommand{\bs}{\boldsymbol{s}}

\newcommand{\bxi}{\boldsymbol{\xi}}
\newcommand{\balpha}{\boldsymbol{\alpha}}
\newcommand{\bbeta}{\boldsymbol{\beta}}
\newcommand{\bgamma}{\boldsymbol{\gamma}}

\begin{document}

\title{A High-Order Surface Finite Element Method Based on Intersections with Background Tetrahedral Meshes}

\author{
Zibo Zhao \thanks{Qiu Zhen College, Tsinghua University, Beijing, China \email{zhaozb21@mails.tsinghua.edu.cn}.}\and
Zuoqiang Shi \thanks{Yau Mathematical Sciences Center, Tsinghua University,
Beijing, China, 100084. \&
Yanqi Lake Beijing Institute of Mathematical Sciences and Applications,
 Beijing, China, 101408.
\email{zqshi@tsinghua.edu.cn}}}
\maketitle

\begin{abstract}
This paper develops a high-order surface finite element method for elliptic equations posed on a smooth closed surface implicitly defined as the zero level set of a function in three dimensions. In contrast to classical surface finite element methods that start from a prescribed triangulation of the surface, the proposed method constructs the discrete surface space from the intersections between the exact surface and an ambient tetrahedral mesh. More precisely, an active tetrahedral shell is generated around the implicit surface, and each cut tetrahedron contributes either a triangular or a quadrilateral surface patch according to its intersection pattern with the zero level set. The finite element space is first defined on the resulting piecewise planar cut surface and is then lifted to the exact surface by local level-set-based parameterizations.
The resulting method combines features of surface FEM and unfitted/trace methods. Like surface FEM, it produces a conforming finite element space on the exact surface after lifting; however, the geometry and finite element space are induced by intersections with background tetrahedra like unfitted/trace methods. We prove that the local lifting maps agree pointwise across common faces and assemble into a global homeomorphism.  We also derive explicit formulas for tangent vectors, Gram matrices, mass and stiffness integrals, and prove stability and high-order derivative estimates for the lifting maps.  The latter estimates are formulated in broken Sobolev norms and lead to a Céa-type energy-norm error bound for the proposed high-order lifted surface finite element method.

\end{abstract}


\section{Introduction}
Let $\Gamma\subset\R^3$ be a smooth closed connected surface represented as the zero level set
\[
        \Gamma=\{x\in U:F(x)=0\},
\]
where $U\subset\R^3$ is a fixed polygonal bulk domain and $F\in C^{m+2}(U)$ satisfies $\nabla F\ne0$ on a tubular neighborhood of $\Gamma$.  We consider the model elliptic surface equation
\begin{equation}
\label{eq:model-pde}
        -\Delta_\Gamma u+u=f\qquad \hbox{on }\Gamma,
\end{equation}
with weak formulation: find $u\in H^1(\Gamma)$ such that
\begin{equation}
\label{eq:weak}
        a(u,v):=\int_\Gamma \nabla_\Gamma u\cdot\nabla_\Gamma v\,\d S+\int_\Gamma uv\,\d S
        =\int_\Gamma fv\,\d S\qquad \forall v\in H^1(\Gamma).
\end{equation}
The goal is to construct a finite element method that uses polynomial spaces on a simple piecewise linear surface, but integrates and differentiates the lifted functions on the exact level-set surface.

The surface finite element method began with Dziuk's finite element method for the Beltrami operator on arbitrary surfaces \cite{Dziuk1988}.  In this classical approach a smooth surface is approximated by a shape-regular triangulated surface, the finite element space consists of continuous piecewise polynomials on the discrete surface, and the weak formulation uses the tangential gradient on the approximating surface. Dziuk and Elliott \cite{DziukElliott2013} gives a broad account of this viewpoint and of related methods for surface PDEs, including triangulated surface methods, implicit surface methods, unfitted finite element methods, and diffuse-interface techniques.  Surface finite element ideas have also been extended to time-dependent surfaces, for instance through the evolving surface finite element method \cite{DziukElliott2007} and through parabolic surface finite element schemes on stationary or moving surfaces \cite{DziukElliott2007Parabolic}.  Eulerian formulations for parabolic equations on implicit surfaces were studied in \cite{DziukElliott2008,DziukElliott2010}, and adaptive surface finite element techniques for implicitly defined surfaces were developed in \cite{DemlowDziuk2007}.

A central lesson from higher-order surface finite elements is that high-order PDE accuracy requires a correspondingly accurate approximation of the geometry.  Demlow \cite{Demlow2009} constructed higher-order analogues of Dziuk's method and proved error estimates in both pointwise and $L^2$-based norms for elliptic problems on surfaces.  His analysis separates the finite-dimensional approximation error from the geometric error, and shows that the degree of the surface approximation must be chosen consistently with the degree of the finite element space in order to obtain the expected convergence rate.  This principle is essential for the present paper.  We use polynomials on a piecewise linear surface for the degrees of freedom, but the Galerkin integrals and surface gradients are evaluated after an exact level-set lift.  In this way the local geometry used in assembly is not merely first-order.

A different family of methods starts from an implicit description of the surface.  The level-set framework for variational problems and PDEs on implicit surfaces was introduced by Bertalmio, Cheng, Osher, and Sapiro \cite{BertalmioChengOsherSapiro2001}; related implicit-surface or embedding methods include the work of Greer, Bertozzi, and Sapiro on fourth-order equations on general geometries \cite{GreerBertozziSapiro2006}, Burger's finite element method on implicit surfaces \cite{Burger2009}, and the closest-point embedding method of Ruuth and Merriman \cite{RuuthMerriman2008}.  Narrow-band and unfitted bulk finite element methods, including the $h$-narrow band method \cite{DeckelnickDziukElliottHeine2010} and unfitted bulk-mesh surface FEM \cite{DeckelnickElliottRanner2014}, are also designed to exploit an ambient mesh rather than a pre-existing fitted surface triangulation.

The trace finite element method is the unfitted approach most closely related to the present work.  Olshanskii, Reusken, and Grande \cite{OlshanskiiReuskenGrande2009} introduced a method in which finite element functions are defined on a background tetrahedral mesh and restricted to an approximation of the surface.  The review of Olshanskii and Reusken \cite{OlshanskiiReusken2018} presents TraceFEM as a general method for stationary and evolving surfaces: the surface may cut the background mesh arbitrarily, while the approximation space is the trace of a bulk finite element space.  This flexibility is especially useful for level-set surfaces and for evolving geometries.  However, if the algebraic system is represented by the background nodal basis, the restricted functions are generally a frame rather than a basis for the trace space.  Consequently the mass matrix may be singular and the stiffness matrix may be ill conditioned or strongly dependent on the cut position.  Matrix properties and conditioning of TraceFEM were studied in \cite{OlshanskiiReusken2010,Reusken2015}.  Stabilization techniques such as ghost penalties \cite{Burman2010Ghost}, stabilized cut finite element methods \cite{BurmanHansboLarson2015}, and full-gradient stabilizations \cite{BurmanHansboLarsonMassingZahedi2016} are therefore important components of practical trace and cut finite element methods.  Cut finite element techniques on embedded manifolds of higher codimension are discussed, for example, in \cite{BurmanHansboLarsonMassing2018}.

High-order TraceFEM combines the trace idea with high-order geometry recovery.  Grande and Reusken \cite{GrandeReusken2016} studied a higher-order finite element method for PDEs on level-set surfaces based on a parametric mapping of a piecewise planar surface.  Lehrenfeld \cite{Lehrenfeld2016} developed high-order unfitted finite element methods on level-set domains using isoparametric mappings of the background mesh.  Building on these ideas, Grande, Lehrenfeld, and Reusken \cite{GrandeLehrenfeldReusken2018} analyzed a high-order TraceFEM for PDEs on level-set surfaces, obtaining optimal-order $H^1$ error bounds and studying stabilizations capable of controlling the condition number for high-order trace discretizations.

The method proposed in this paper combines the surface FEM and TraceFEM viewpoints in a different way.  Like surface FEM, it ultimately works with a conforming finite element space on a surface triangulation and lifts this space to the exact surface.  Like TraceFEM, it uses an implicit level-set description and a tetrahedral background structure to recover local geometry.  The mesh generation step first produces a quasi-uniform, shape-regular surface triangulation, for example from a background grid together with a moving-node reconstruction.  A balanced tubular offset tetrahedral shell is then built around this triangulation.  The shell gives an admissible active tetrahedral mesh whose vertices lie on two nearby normal layers, so that four-edge cuts satisfy the normal balance condition needed in the high-order derivative estimates.  Each active tetrahedron gives an explicit local parameterization of the exact surface patch: a vertex-projection map in the three-edge case and a ruled two-rail map in the four-edge case.  Thus the finite element unknowns live on a genuine surface triangulation, avoiding the singular trace-space matrix issue, while the background tetrahedral geometry still supplies explicit level-set-based parameterizations and Jacobians for high-order assembly.

\subsection{Overview of the proposed method}
\label{sec:method-overview}

The proposed method consists of three main steps.

\emph{Step 1: construction of an auxiliary surface mesh.}
We first construct a quasi-uniform and shape-regular triangulation
\(\mathcal S_h\) of the implicit surface \(\Gamma\). Such a triangulation may,
for example, be obtained from an ambient tetrahedral mesh by the
mesh-adjustment and level-set intersection procedure of
Zhao \cite{Zhao2025Quadrature}. This procedure avoids degenerate
intersections and produces a face-compatible piecewise planar reconstruction
of the surface.

\emph{Step 2: construction of a tubular tetrahedral shell.}
Starting from \(\mathcal S_h\), we place two offset vertices
\[
    p_i^\pm=p_i\pm\rho h_\Gamma\nn(p_i)
\]
on the two sides of \(\Gamma\) for every surface vertex \(p_i\).
Each triangle of \(\mathcal S_h\) then generates a triangular prism, and all
such prisms are subdivided by a globally consistent rule into a tetrahedral
mesh \(\mathcal T_h^\Gamma\). The offset distance is comparable to the surface
mesh size, so the resulting tetrahedra remain uniformly shape regular and
satisfy the balanced normal condition required in the four-edge analysis.

The intersections of \(\Gamma\) with the tetrahedra of
\(\mathcal T_h^\Gamma\) generate a new cut surface \(\Gamma_h\). Each cut is
either triangular or quadrilateral; a quadrilateral cut is divided into two
triangles using face-compatible diagonal choices. We denote the resulting
shape-regular surface triangulation by \(\mathcal K_h\). Thus
\(\mathcal S_h\) is used only to construct the tubular shell, whereas the
finite element surface \(\Gamma_h\) is induced by the intersections of the
exact surface with the background tetrahedra.

\emph{Step 3: parameterized surface finite element computation.}
On \(\Gamma_h\), we define
\[
    S_h^r
    =
    \left\{
        \varphi_h\in C^0(\Gamma_h):
        \varphi_h|_K\in\mathbb P_r(K),
        \quad K\in\mathcal K_h
    \right\}.
\]
For every cut tetrahedron \(T\), a local level-set parameterization
\[
    \Phi_T:\Gamma_h\cap T\longrightarrow\Gamma\cap T
\]
maps the planar cut patch to the exact surface. We use a vertex-projection
parameterization for triangular cuts and a ruled two-rail parameterization
for quadrilateral cuts. In the latter case, the same parameterization is used
on both triangles of the planar representation. Consequently, the local maps
are compatible across element interfaces and assemble into a continuous
global map \(\Phi_h\).

The corresponding lifted finite element space is
\[
    S_{h,r}^l
    =
    \left\{
        \varphi_h^l:\Gamma\to\mathbb R:
        \varphi_h^l\circ\Phi_h=\varphi_h,
        \quad \varphi_h\in S_h^r
    \right\}
    \subset H^1(\Gamma).
\]
The discrete problem is therefore posed on the exact surface: find
\(u_h^l\in S_{h,r}^l\) such that
\[
    \int_\Gamma
        \nabla_\Gamma u_h^l\cdot\nabla_\Gamma v_h^l\,dS
    +
    \int_\Gamma u_h^l v_h^l\,dS
    =
    \int_\Gamma f v_h^l\,dS
    \qquad
    \forall v_h^l\in S_{h,r}^l .
\]
The element integrals are evaluated by pulling them back through the local
parameterizations \(\Phi_T\). Hence the finite element connectivity is
provided by the cut surface \(\Gamma_h\), while the metric, surface Jacobian,
and tangential derivatives are computed from the exact level-set geometry.

The following sections describe the tubular shell, the local
parameterizations, and the error analysis in detail. We assume that the
surface integrals are evaluated exactly or by sufficiently high-order
quadrature; the corresponding quadrature errors can be treated by standard
Strang-lemma arguments.

\section{Geometry, adjusted meshes, and cut configurations}
\label{sec:geometry}
Let $d(x)$ denote the signed distance function to $\Gamma$, with unit normal $\nn=\nabla d$ in a tubular neighborhood.  Let $r_\Gamma>0$ be the reach of $\Gamma$, so that the closest-point projection $\pi_\Gamma$ is single-valued whenever $|d(x)|<r_\Gamma$.  Let $\kappa_1,\kappa_2$ be the principal curvatures of $\Gamma$, and put
\[
        K_\Gamma=\frac12\max_{x\in\Gamma}\left(|\kappa_1(x)|+|\kappa_2(x)|\right).
\]
Let $\Th$ be a shape-regular tetrahedral mesh of $U$, with mesh size
\[
        h=\max_{T\in\Th}\diam(T).
\]
Throughout the paper, $C$ denotes a constant independent of $h$ and of the particular cut tetrahedron.

\begin{definition}[Admissible adjusted cut mesh]
\label{def:admissible}
The adjusted mesh $\Th$ is called admissible if the following conditions hold for all sufficiently small $h$.
\begin{enumerate}[(i)]
\item $\Th$ is uniformly shape regular.
\item No vertex lies on $\Gamma$, and every vertex $A$ of a cut tetrahedron satisfies
\[
        |d(A)|\ge \eta h
\]
for a fixed $\eta>0$.
\item If $A_i$ and $A_j$ are two vertices of a tetrahedron, then
\begin{equation}
\label{eq:balanced-distance}
        |d(A_i)|-|d(A_j)|=O(h^2).
\end{equation}
\end{enumerate}
Condition \eqref{eq:balanced-distance} is called the balanced normal adjustment condition.
\end{definition}

The first two parts of Definition \ref{def:admissible} are enough to ensure that a line segment inside one tetrahedron does not cut $\Gamma$ more than once, as shown in Theorem \ref{thm:cut-configurations}. The last part is used only in the high-order proof for the four-edge cut.  In the numerical method and in the constructive interpretation of the assumptions, we do not use a vertex-only snapping of an arbitrary Cartesian or tetrahedral mesh to the two layers $d=\pm \rho h$.  Such a snapping may destroy shape regularity when a same-sign edge is almost normal to the surface.  Instead we use the following tubular offset shell construction.

\begin{algorithm}[t]
\caption{Tubular offset shell mesh}
\label{alg:tubular-offset}
\begin{algorithmic}[1]
\State Construct a quasi-uniform, shape-regular surface triangulation $\mathcal S_h$ of $\Gamma$, with vertices $p_i$ and triangles $K_\sigma=[p_{\sigma_1},p_{\sigma_2},p_{\sigma_3}]$.  In practice $\mathcal S_h$ may be obtained from an ambient tetrahedral mesh by the mesh-adjustment and level-set intersection construction of \cite{Zhao2025Quadrature}.
\State Let $h_\Gamma=\max_{K_\sigma\in\mathcal S_h}\diam(K_\sigma)$ and choose a fixed $\rho>0$ such that $\rho h_\Gamma<r_\Gamma/4$.
\State For each surface vertex set
\[
        p_i^-=p_i-\rho h_\Gamma \nn(p_i),
        \qquad
        p_i^+=p_i+\rho h_\Gamma \nn(p_i).
\]
If the signed distance is not explicitly available, the same step may be implemented by solving along the normal or gradient line for two nearby level surfaces $F=\pm\sigma h_\Gamma$; since $F=a d$ in the tubular neighborhood with $a$ bounded away from zero and Lipschitz, this gives the same balanced condition up to $O(h_\Gamma^2)$.
\State For every surface triangle $K_\sigma=[p_1,p_2,p_3]$, form the triangular prism with vertices $p_1^-,p_2^-,p_3^-,p_1^+,p_2^+,p_3^+$ and split it by a fixed shape-regular prism subdivision, for example
\[
\begin{aligned}
        &\operatorname{conv}\{p_1^-,p_2^-,p_3^-,p_1^+\},\\
        &\operatorname{conv}\{p_2^-,p_3^-,p_1^+,p_2^+\},\\
        &\operatorname{conv}\{p_3^-,p_1^+,p_2^+,p_3^+\}.
\end{aligned}
\]
\State The union of these tetrahedra is the active adjusted mesh.  If a full bulk mesh in $U$ is required, the exterior of the shell can be filled by any uniformly shape-regular tetrahedral mesh; the analysis below only uses the active cut shell.
\end{algorithmic}
\end{algorithm}

This construction gives an admissible adjusted mesh for sufficiently small $h$.  Indeed, the normal coordinate map $p\mapsto p\pm\rho h_\Gamma\nn(p)$ has differential $(I\mp\rho h_\Gamma S)$ on tangent vectors, where $S$ is the Weingarten map.  Since $\rho h_\Gamma\|S\|_{L^\infty}<1/2$ for small $h$, this map is uniformly bi-Lipschitz on every surface triangle.  The tangential dimensions of each prism are comparable to $h_\Gamma$, and the normal thickness is $2\rho h_\Gamma$, hence the fixed prism subdivision produces uniformly shape-regular tetrahedra.  Moreover,
\[
        d(p_i^+)=\rho h_\Gamma,
        \qquad
        d(p_i^-)=-\rho h_\Gamma,
\]
so same-sign vertices in one shell tetrahedron have equal signed distance.  If the offset layers are constructed by fixing $F=\pm\sigma h_\Gamma$ rather than $d=\pm\rho h_\Gamma$, then the identity $F=a d$ with $a$ Lipschitz and $a\ge c>0$ implies $d(A_i)-d(A_j)=O(h_\Gamma^2)$ for same-sign vertices in the same tetrahedron.  The vertex separation required in Definition \ref{def:admissible}(ii) follows at once from the fixed offset thickness.

Next, we give two technical results which are important in the construction and analysis of the proposed method.
\begin{lemma}[One-dimensional non-reintersection]
\label{lem:one-segment}
Assume $hK_\Gamma\le1/4$ and $h<r_\Gamma$.  Let $A,B\in U$ with $|A-B|\le h$.  If the segment $AB$ intersects $\Gamma$ at least twice, then
\[
        \max\{|d(A)|,|d(B)|\}\le K_\Gamma h^2.
\]
Consequently, if the endpoints of a segment in an admissible cut tetrahedron have signed distances bounded below by $\eta h$ and $h$ is sufficiently small, that segment intersects $\Gamma$ at most once.
\end{lemma}

\begin{proof}
Let $x(t)=A+t(B-A)/|B-A|$, $0\le t\le |B-A|$, and set $\delta(t)=d(x(t))$.  In the tubular neighborhood, the closest-point formula
\[
        x(t)=\pi_\Gamma(x(t))+\delta(t)\nn(\pi_\Gamma(x(t)))
\]
can be differentiated.  If $v=(B-A)/|B-A|$ and $S$ is the Weingarten map, then
\[
        \delta'(t)=v\cdot \nn(\pi_\Gamma(x(t))),\qquad
        \delta''(t)=-v\cdot S(I-\delta(t)S)^{-1}v_{\rm tan}.
\]
Since $|\delta(t)|\le h$ and $hK_\Gamma\le1/4$, we have $|\delta''(t)|\le 2K_\Gamma$.  If $\delta$ has two zeros $t_1,t_2$, the interpolation remainder formula gives
\[
        |\delta(t)|\le \frac12\max |\delta''|\, |(t-t_1)(t-t_2)|\le K_\Gamma h^2,
\]
and the result follows by taking $t=0$ and $t=|B-A|$.
\end{proof}
From Lemma \ref{lem:one-segment}, it is straightforward to obtain the cut configurations between the surface and the tetrahedrons. 
\begin{theorem}[Cut configurations]
\label{thm:cut-configurations}
Let $\Th$ be admissible and $h$ sufficiently small.  For every tetrahedron $T=\operatorname{conv}\{x_0,x_1,x_2,x_3\}$, the intersection $\Gamma\cap T$ is one of the following types.
\begin{enumerate}[(i)]
\item If all $F(x_i)$ have the same sign, then $\Gamma\cap T=\emptyset$.
\item If exactly one vertex has sign different from the other three, then $\Gamma\cap T$ is a single triangular patch, and every segment from the exceptional vertex to the opposite face intersects $\Gamma$ at most once.
\item If two vertices are on each side of $\Gamma$, then $\Gamma\cap T$ is a single quadrilateral patch.  If, after relabeling,
\[
        F(x_0),F(x_2)>0,
        \qquad F(x_1),F(x_3)<0,
\]
then for every $M\in x_0x_2$ and $N\in x_1x_3$, the segment $MN$ intersects $\Gamma$ exactly once.
\end{enumerate}
\end{theorem}

\section{Local parameterizations and Jacobians}
\label{sec:local-parameterizations}
This section gives the exact local maps used to lift the finite element functions.  All local maps are from a piecewise linear approximation $\Gamma_h\cap T$ to the exact patch $\Gamma\cap T$.

\subsection{Triangular cut: vertex projection}
\label{subsec:case1}
Assume $F(x_0)$ has sign opposite to $F(x_1),F(x_2),F(x_3)$.  Let $B_i$ be the unique intersection of $\Gamma$ with the edge $x_0x_i$, $i=1,2,3$.  The linear patch is the triangle
\[
        K_h=\operatorname{conv}\{B_1,B_2,B_3\}.
\]
For $X\in K_h$, define $Y=\Phi_T(X)$ to be the unique point of
\(\Gamma\) on the ray starting at \(x_0\) and passing through \(X\).
The cut-configuration theorem implies that this point lies in \(T\).
If $X=X(\mu_1,\mu_2)$ is the affine map from the reference triangle to $K_h$, then $Y=Y(\mu_1,\mu_2)$ is characterized by
\begin{equation}
\label{eq:case1-implicit}
        F(Y)=0,
        \qquad
        (Y-x_0)\times (X(\mu_1,\mu_2)-x_0)=0.
\end{equation}
\begin{figure}
    \centering
    \includegraphics[width=0.4\linewidth]{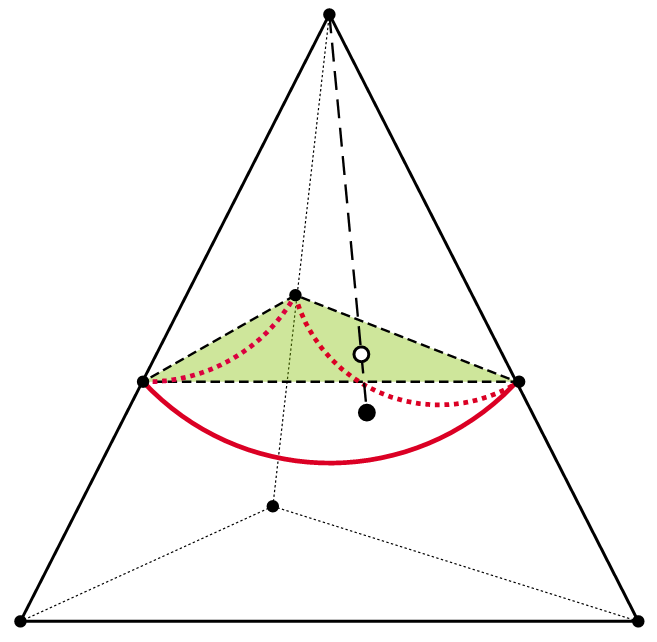}
    \caption{Illustration of the parametrization in Section \ref{subsec:case1}.}
    \label{fig:para-1-3}
\end{figure}
Equivalently, one may use two independent scalar equations from the cross product.  The implicit function theorem gives
\begin{equation}
\label{eq:case1-jacobian}
        J_T(\mu_1,\mu_2)=\frac{\partial Y}{\partial(\mu_1,\mu_2)}
        =-\left(\frac{\partial H}{\partial Y}\right)^{-1}
          \frac{\partial H}{\partial(\mu_1,\mu_2)},
\end{equation}
where $H$ denotes the three equations in \eqref{eq:case1-implicit}.  The $3\times2$ matrix $J_T$ is the local surface Jacobian, and the first fundamental form is
\[
        G_T=J_T^TJ_T.
\]

\subsection{Quadrilateral cut: ruled two-rail parameterization}
\label{subsec:case2}
Assume, after relabeling, that
\[
        F(x_0),F(x_2)>0,
        \qquad
        F(x_1),F(x_3)<0.
\]
The two same-sign edges are the rails
\begin{equation}
\label{eq:rails}
        M(u)=(1-u)x_0+ux_2,
        \qquad
        N(v)=(1-v)x_1+vx_3,
        \qquad 0\le u,v\le1.
\end{equation}
For every $(u,v)\in[0,1]^2$, Theorem \ref{thm:cut-configurations} gives a unique $w=w(u,v)\in(0,1)$ such that
\begin{equation}
\label{eq:ruled-param}
        Y(u,v)=(1-w(u,v))M(u)+w(u,v)N(v)\in\Gamma.
\end{equation}
Thus
\begin{equation}
\label{eq:w-equation}
        F\big((1-w)M(u)+wN(v)\big)=0.
\end{equation}
Let
\[
        V(u,v)=N(v)-M(u),
        \qquad n_F(Y)=\nabla F(Y).
\]
Since $M(u)$ and $N(v)$ stay on opposite sides of $\Gamma$ with signed distances comparable to $h$, we have
\begin{equation}
\label{eq:transversality}
        |n_F(Y)\cdot V(u,v)|\ge c h.
\end{equation}
Implicit differentiation of \eqref{eq:w-equation} gives
\begin{equation}
\label{eq:w-derivatives}
        w_u=-\frac{(1-w)n_F(Y)\cdot M'(u)}{n_F(Y)\cdot V(u,v)},
        \qquad
        w_v=-\frac{w\,n_F(Y)\cdot N'(v)}{n_F(Y)\cdot V(u,v)}.
\end{equation}
Consequently the tangent vectors are
\begin{equation}
\label{eq:case2-tangent}
\begin{aligned}
        Y_u&=(1-w)\left[M'(u)-\frac{n_F(Y)\cdot M'(u)}{n_F(Y)\cdot V(u,v)}V(u,v)\right],\\
        Y_v&=w\left[N'(v)-\frac{n_F(Y)\cdot N'(v)}{n_F(Y)\cdot V(u,v)}V(u,v)\right].
\end{aligned}
\end{equation}
The local Jacobian for the ruled square is $J=[Y_u,Y_v]$ and $G=J^TJ$.

\begin{figure}
    \centering
    \includegraphics[width=0.4\linewidth]{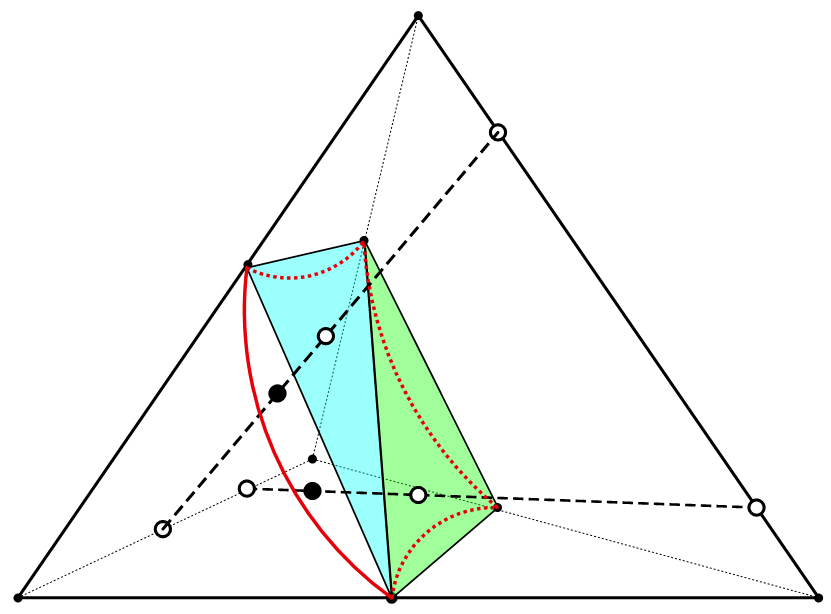}
    \caption{Illustration of the parametrization in Section \ref{subsec:case2}.}
    \label{fig:para-2-2}
\end{figure}

For finite elements, the four edge-intersection points form a generally
non-coplanar spatial quadrilateral, which is represented by two planar
triangles.  Let $P_{ij}$ denote the zero of $F$ on the edge $x_ix_j$.
We choose the diagonal \(P_{12}P_{30}\) and set
\[
        K_{1,h}=\triangle P_{01}P_{12}P_{30},
        \qquad
        K_{2,h}=\triangle P_{12}P_{23}P_{30}.
\]
For $X\in K_{1,h}\cup K_{2,h}$ with tetrahedral barycentric coordinates $(\alpha,\beta,\gamma,\delta)$ relative to $(x_0,x_1,x_2,x_3)$, define
\begin{equation}
\label{eq:ratio-map}
        u(X)=\frac{\gamma}{\alpha+\gamma},
        \qquad
        v(X)=\frac{\delta}{\beta+\delta},
        \qquad
        w_h(X)=\beta+\delta.
\end{equation}
Then
\begin{equation}
\label{eq:X-MN}
        X=(1-w_h(X))M(u(X))+w_h(X)N(v(X)).
\end{equation}
The lift on the two triangles is
\begin{equation}
\label{eq:case2-lift}
        \Phi_T(X)=Y(u(X),v(X)).
\end{equation}

\subsection{Local bijectivity and compatibility of the global lift}
\label{subsec:global-lift}

We next prove that the local maps are bijective and agree pointwise on
common faces.  The balanced normal condition is not needed for this
topological result; it enters only in the uniform metric and high-order
derivative estimates below.

\begin{lemma}[Bijectivity of the triangular lift]
\label{lem:triangular-bijection}
The vertex-projection map
\[
        \Phi_T:K_h\longrightarrow\Gamma\cap T
\]
in a three-edge cut is a homeomorphism.
\end{lemma}

\begin{proof}
Let \(F_0=\operatorname{conv}\{x_1,x_2,x_3\}\) be the face opposite
\(x_0\), and write
\[
        B_i=x_0+t_i(x_i-x_0),\qquad 0<t_i<1,\qquad i=1,2,3.
\]
For
\[
        X=\sum_{i=1}^3\lambda_iB_i,\qquad
        \lambda_i\ge0,\qquad \sum_i\lambda_i=1,
\]
set
\[
        s(X)=\sum_{i=1}^3\lambda_it_i,
        \qquad
        P(X)=\sum_{i=1}^3\frac{\lambda_it_i}{s(X)}x_i.
\]
Then \(P(X)\in F_0\) and
\[
        X=x_0+s(X)(P(X)-x_0).
\]
Thus \(P(X)\) is the intersection of the ray \(x_0X\) with \(F_0\).
Conversely, if \(P=\sum_i\mu_ix_i\in F_0\), then
\[
        s(P)=\left(\sum_{i=1}^3\frac{\mu_i}{t_i}\right)^{-1},
        \qquad
        \lambda_i(P)=\frac{s(P)\mu_i}{t_i}
\]
give the unique \(X\in K_h\) on the ray \(x_0P\).  Hence radial
projection from \(x_0\) is a homeomorphism from \(K_h\) to \(F_0\).

By Theorem \ref{thm:cut-configurations}, every segment \(x_0P\), with
\(P\in F_0\), has exactly one intersection with \(\Gamma\), and this
intersection depends continuously on \(P\).  Radial projection from the
exact patch \(\Gamma\cap T\) to \(F_0\) is therefore also a
homeomorphism.  The map \(\Phi_T\) is the composition of the first
radial projection with the inverse of the second one, and the result
follows.
\end{proof}

\begin{lemma}[Bijectivity of the ratio map]
\label{lem:ratio-bijection}
For a four-edge cut, the map
\[
        \mathcal R_T:
        K_{1,h}\cup K_{2,h}\longrightarrow[0,1]^2,
        \qquad
        \mathcal R_T(X)=(u(X),v(X)),
\]
defined in \eqref{eq:ratio-map}, is a homeomorphism.  Consequently,
\(\Phi_T=Y\circ\mathcal R_T\) is a homeomorphism from
\(K_{1,h}\cup K_{2,h}\) to \(\Gamma\cap T\).
\end{lemma}

\begin{proof}
Introduce
\[
        B(u,v,w)=(1-w)M(u)+wN(v).
\]
The tetrahedral barycentric coordinates of this point are
\[
        \alpha=(1-w)(1-u),\quad
        \beta=w(1-v),\quad
        \gamma=(1-w)u,\quad
        \delta=wv.
\]
For \(0<w<1\), these relations can be inverted uniquely:
\begin{equation}
\label{eq:barycentric-ratio-inverse}
        w=\beta+\delta,\qquad
        u=\frac{\gamma}{\alpha+\gamma},\qquad
        v=\frac{\delta}{\beta+\delta}.
\end{equation}
Theorem \ref{thm:cut-configurations} gives a unique root
\(w(u,v)\in(0,1)\) on every segment \(M(u)N(v)\).  It follows from
\eqref{eq:barycentric-ratio-inverse} that
\[
        Y:[0,1]^2\longrightarrow\Gamma\cap T
\]
is one-to-one and onto.  Transversality \eqref{eq:transversality}
implies continuity of \(w\), so \(Y\) is a homeomorphism.

It remains to prove the same assertion for \(\mathcal R_T\).  Write the
four edge roots in the form
\[
\begin{aligned}
        P_{01}&=(1-a)x_0+ax_1,&
        P_{12}&=(1-b)x_2+bx_1,\\
        P_{23}&=(1-c)x_2+cx_3,&
        P_{30}&=(1-d)x_0+dx_3,
\end{aligned}
\qquad 0<a,b,c,d<1.
\]
On \(K_{1,h}\), put
\[
        X=(1-p-q)P_{01}+pP_{12}+qP_{30},
        \qquad p,q\ge0,\quad p+q\le1.
\]
Let \(D_+=\alpha+\gamma\) and \(D_-=\beta+\delta\).  Direct
differentiation gives
\begin{equation}
\label{eq:ratio-Jacobian-first}
\det\frac{\partial(u,v)}{\partial(p,q)}
=
\frac{(1-b)d}{D_+^2D_-^2}
\left[
(1-p-q)a(1-a)+pb(1-a)+qa(1-d)
\right]>0.
\end{equation}
On \(K_{2,h}\), put
\[
        X=(1-p-q)P_{12}+pP_{23}+qP_{30}.
\]
The analogous calculation gives
\begin{equation}
\label{eq:ratio-Jacobian-second}
\det\frac{\partial(u,v)}{\partial(p,q)}
=
\frac{b(1-d)}{D_+^2D_-^2}
\left[
(1-p-q)c(1-b)+pc(1-c)+qd(1-c)
\right]>0.
\end{equation}
Thus \(\mathcal R_T\) is an orientation-preserving local
homeomorphism in the interior of each triangle.

The four exterior edges are mapped, respectively, to \(v=0\),
\(u=0\), \(u=1\), and \(v=1\).  On the common diagonal, write
\[
        X(t)=tP_{12}+(1-t)P_{30},\qquad 0\le t\le1.
\]
Then
\[
        u(t)=
        \frac{t(1-b)}
        {(1-t)(1-d)+t(1-b)},
        \qquad
        v(t)=
        \frac{(1-t)d}
        {tb+(1-t)d}.
\]
The first function is strictly increasing and the second is strictly
decreasing.  Hence the diagonal is mapped to a simple curve joining
\((0,1)\) to \((1,0)\).  Equations
\eqref{eq:ratio-Jacobian-first}--\eqref{eq:ratio-Jacobian-second},
together with the one-to-one boundary maps, show by the planar global
inverse theorem (equivalently, by the Brouwer degree, to which every
preimage contributes \(+1\)) that the two triangles map
homeomorphically to the two regions separated by this curve.  Their
images have disjoint interiors and their union is \([0,1]^2\).
Therefore \(\mathcal R_T\) is a homeomorphism.  Since \(Y\) is a
homeomorphism, so is \(\Phi_T=Y\circ\mathcal R_T\).
\end{proof}

We now introduce a construction intrinsic to a cut face.  Let
\[
        \mathcal F=\operatorname{conv}\{e,z_1,z_2\}
\]
be a mesh face on which \(e\) has sign opposite to \(z_1,z_2\), and let
\(Q_i=ez_i\cap\Gamma\).  The discrete trace on this face is the segment
\(L_{\mathcal F,h}=Q_1Q_2\).  For \(X\in L_{\mathcal F,h}\), let
\(Z(X)\in z_1z_2\) be the intersection of the ray \(eX\) with the
opposite edge.  Define the canonical face lift
\begin{equation}
\label{eq:canonical-face-lift}
        \Phi_{\mathcal F}(X)
        =
        \text{the unique point of }eZ(X)\cap\Gamma.
\end{equation}
The two-dimensional version of Lemma \ref{lem:triangular-bijection}
shows that \(\Phi_{\mathcal F}\) is a homeomorphism from
\(L_{\mathcal F,h}\) to \(\Gamma\cap\mathcal F\).

\begin{theorem}[Global lift homeomorphism]
\label{thm:global-lift-homeomorphism}
Assume that the active tetrahedra form a conforming tetrahedral complex
covering the closed surface \(\Gamma\), and that the hypotheses of
Theorem \ref{thm:cut-configurations} hold.  Then the local maps
\(\Phi_T\) agree pointwise on common faces and define a continuous
bijection
\[
        \Phi_h:\Gamma_h\longrightarrow\Gamma.
\]
In particular, \(\Phi_h\) is a homeomorphism and
\[
        S_{h,r}^l\subset H^1(\Gamma).
\]
\end{theorem}

\begin{proof}
Consider first a three-edge cut.  Every cut face contains the
exceptional vertex, and the restriction of the vertex-projection map to
that face is exactly \(\Phi_{\mathcal F}\) in
\eqref{eq:canonical-face-lift}.  The opposite face has vertices of one
sign and does not meet \(\Gamma\).

For a four-edge cut with the labeling in
\eqref{eq:rails}, the restrictions to the four faces are
\[
\begin{array}{c|c|c}
\mathcal F & \text{vanishing barycentric coordinate}
            & \text{restricted ratio variable}\\ \hline
x_0x_1x_2 & \delta=0 & v=0\\
x_1x_2x_3 & \alpha=0 & u=1\\
x_0x_2x_3 & \beta=0  & v=1\\
x_0x_1x_3 & \gamma=0 & u=0 .
\end{array}
\]
For example, on \(x_0x_1x_2\) the ruled segment is
\(x_1M(u)\), so \(Y(u,0)\) is precisely the canonical face lift.
The other three cases are identical.  Consequently, for two
tetrahedra \(T^+\) and \(T^-\) sharing a cut face \(\mathcal F\),
\begin{equation}
\label{eq:pointwise-face-compatibility}
        \Phi_{T^+}|_{\Gamma_h\cap\mathcal F}
        =
        \Phi_{\mathcal F}
        =
        \Phi_{T^-}|_{\Gamma_h\cap\mathcal F}.
\end{equation}
The two triangles inside a four-edge cut also agree pointwise on their
artificial diagonal because both use the same ratio map and the same
ruled map \(Y\).  The gluing lemma and
\eqref{eq:pointwise-face-compatibility} therefore give a continuous
global map \(\Phi_h\).

To prove surjectivity, take \(y\in\Gamma\).  It belongs to
\(\Gamma\cap T\) for at least one active tetrahedron \(T\), and Lemmas
\ref{lem:triangular-bijection} and \ref{lem:ratio-bijection} give a
local preimage in \(\Gamma_h\cap T\).

For injectivity, suppose that \(\Phi_h(X)=\Phi_h(X')=y\).  If \(y\)
lies in the interior of a tetrahedron, both preimages lie in the
corresponding local cut patch and local injectivity gives \(X=X'\).
If \(y\) lies in the relative interior of a mesh face, every preimage
lies on the common discrete trace and the injectivity of
\(\Phi_{\mathcal F}\) gives the same conclusion.  If \(y\) lies on a
mesh edge, the edge has only one intersection with \(\Gamma\), and all
local maps fix that edge-intersection node.  A mesh vertex cannot lie
on \(\Gamma\) by admissibility.  These cases exhaust the mesh
skeleton, so \(\Phi_h\) is injective.

Finally, \(\Gamma_h\) is compact and \(\Gamma\subset\mathbb R^3\) is
Hausdorff.  A continuous bijection from a compact space to a Hausdorff
space has a continuous inverse, proving that \(\Phi_h\) is a
homeomorphism.  The lifted functions are continuous and piecewise
\(H^1\) on the exact surface, hence belong to \(H^1(\Gamma)\).
\end{proof}

\subsection{Surface gradients and element matrices}
Let $r: \wh K\to\Gamma\cap T$ be either of the local parameterizations above, with columns $r_1,r_2$ in its Jacobian $J_r$.  Put
\[
        G_r=J_r^TJ_r=\begin{pmatrix}r_1\cdot r_1&r_1\cdot r_2\\ r_2\cdot r_1&r_2\cdot r_2\end{pmatrix}.
\]
If $\phi$ and $\psi$ are scalar functions on the parameter domain, their lifted surface gradients satisfy
\begin{equation}
\label{eq:surface-gradient-formula}
        \nabla_\Gamma \phi^l\cdot \nabla_\Gamma \psi^l
        =
        \begin{pmatrix}\phi_{\xi_1}&\phi_{\xi_2}\end{pmatrix}
        G_r^{-1}
        \begin{pmatrix}\psi_{\xi_1}\\\psi_{\xi_2}\end{pmatrix}.
\end{equation}
Thus each element contribution is
\begin{equation}
\label{eq:local-bilinear}
\begin{aligned}
        \int_{\Gamma\cap T}\phi^l\psi^l\,\d S
        &=\int_{\wh K}\phi\psi\sqrt{\det G_r}\,\d\bxi,\\
        \int_{\Gamma\cap T}\nabla_\Gamma\phi^l\cdot\nabla_\Gamma\psi^l\,\d S
        &=\int_{\wh K}
        \begin{pmatrix}\phi_{\xi_1}&\phi_{\xi_2}\end{pmatrix}
        G_r^{-1}
        \begin{pmatrix}\psi_{\xi_1}\\\psi_{\xi_2}\end{pmatrix}
        \sqrt{\det G_r}\,\d\bxi.
\end{aligned}
\end{equation}
These are the formulas used in assembly.  A high-order Gaussian rule on the reference triangle or square can be used, because the parameterized integrands are smooth on each element.

\section{Finite element space and Galerkin method}
\label{sec:fem}
Let $\Gamma_h$ be the union of all linear surface elements generated from the cut tetrahedra.  In the four-edge case, the quadrilateral is divided into the two triangles described above.  Let $\Kh$ denote this surface triangulation.  For a polynomial degree $r\ge1$, define
\begin{equation}
\label{eq:Sh}
        \Sh^r=\left\{\phi_h\in C^0(\Gamma_h):\; \phi_h|_K\in \mathbb{P}_r(K),\quad K\in\Kh\right\}.
\end{equation}
Let $\Phi_h:\Gamma_h\to\Gamma$ be the piecewise map defined by \eqref{eq:case1-implicit} in triangular cuts and by \eqref{eq:case2-lift} in quadrilateral cuts.  The lifted finite element space is
\begin{equation}
\label{eq:lifted-space}
        S_{h,r}^l=\left\{\phi_h^l:\Gamma\to\R:\; \phi_h^l(\Phi_h(X))=\phi_h(X),\quad \phi_h\in\Sh^r\right\}.
\end{equation}
Theorem \ref{thm:global-lift-homeomorphism} implies
$S_{h,r}^l\subset H^1(\Gamma)$.

The exact-surface finite element method is: find $u_h^l\in S_{h,r}^l$ such that
\begin{equation}
\label{eq:discrete}
        a(u_h^l,v_h^l)=\int_\Gamma f v_h^l\,\d S
        \qquad \forall v_h^l\in S_{h,r}^l.
\end{equation}
Because $S_{h,r}^l\subset H^1(\Gamma)$, this is a conforming Galerkin method.  Therefore, by Cea's lemma,
\begin{equation}
\label{eq:cea}
        \norm{u-u_h^l}_{H^1(\Gamma)}
        \le C\inf_{v_h^l\in S_{h,r}^l}\norm{u-v_h^l}_{H^1(\Gamma)}.
\end{equation}
The rest of the analysis is devoted to estimating the approximation term on the right-hand side.

\subsection{Stability of the lift}
\label{sec:lift-stability}
The first estimate needed for \eqref{eq:cea} is the uniform stability of the lift in $L^2$ and $H^1$.

\begin{lemma}[Uniform metric bounds]
\label{lem:metric-bounds}
For every surface element $K\in\Kh$, let $r_K:\wh K\to\Gamma\cap T$ be its exact parameterization and $X_K:\wh K\to K$ the affine parameterization of the linear element.  Then, for sufficiently small $h$,
\begin{equation}
\label{eq:metric-equivalence}
        c h^2\le \lambda_{\min}(J_{r_K}^TJ_{r_K})\le \lambda_{\max}(J_{r_K}^TJ_{r_K})\le C h^2,
\end{equation}
and the same bounds hold for $J_{X_K}^TJ_{X_K}$.
\end{lemma}

\begin{proof}
We first prove the estimate for the affine map $X_K$.  Since the induced surface triangulation $\Kh$ is uniformly shape regular and every element has diameter comparable to $h$, the two singular values of the $3\times2$ matrix $J_{X_K}$ are bounded between $c h$ and $C h$.  Equivalently,
\[
        c h^2 |\zeta|^2
        \le \zeta^T J_{X_K}^T J_{X_K}\zeta
        = |J_{X_K}\zeta|^2
        \le C h^2 |\zeta|^2,
        \qquad \zeta\in\R^2.
\]
This proves the metric bounds for $X_K$.

It remains to compare the exact parameterization $r_K$ with the affine one.  We give the details for the two local constructions.

In the triangular case write, for $X=X_K(\bxi)$,
\[
        V(\bxi)=\frac{X(\bxi)-x_0}{h},
        \qquad
        Y(\bxi)=x_0+k(\bxi)V(\bxi),
\]
where $k=O(h)$.  The vector $V$ has a normal component bounded away from zero at $Y$.  Indeed, $d(Y)=0$ and Taylor expansion at $Y$ gives
\[
        d(x_0)=d(Y-kV)=-k\,\nn(Y)\cdot V+O(h^2),
\]
while $|d(x_0)|\ge\eta h$ and $k\simeq h$.  Hence $|\nn(Y)\cdot V|\ge c$.  Differentiating the equation $F(x_0+kV)=0$ gives
\[
        k_i=-k\frac{\nabla F(Y)\cdot V_i}{\nabla F(Y)\cdot V},
        \qquad i=1,2,
\]
where $V_i=\partial_{\xi_i}V=O(1)$.  Therefore
\[
        Y_i=k_iV+kV_i=O(h),
        \qquad i=1,2.
\]
This gives the upper bound $\lambda_{\max}(J_{r_K}^TJ_{r_K})\le C h^2$.  For the lower bound, observe that $Y_i$ is obtained from $kV_i$ by projection onto $T_Y\Gamma$ along the transverse direction $V$.  The projection operator is uniformly bounded above and below on the two-dimensional space spanned by $V_1,V_2$, because $V$ is transverse to $T_Y\Gamma$ and the triangle $K$ is shape regular.  Thus for every $\zeta\in\R^2$,
\[
        |J_{r_K}\zeta|\ge c h |\zeta|,
\]
which proves the desired lower metric bound in the triangular case.

In the quadrilateral case, use the notation of Section \ref{subsec:case2}.  Put $e_M=(x_2-x_0)/h$, $e_N=(x_3-x_1)/h$, and $v_r=(N-M)/h$.  Shape regularity of the shell tetrahedra gives $|e_M|+|e_N|+|v_r|\le C$ and prevents the three directions from becoming linearly dependent in a way that would collapse the prism.  The transversality estimate \eqref{eq:transversality} is equivalent to $|\nabla F(Y)\cdot v_r|\ge c$.  Formula \eqref{eq:case2-tangent} can be written as
\[
\begin{aligned}
        Y_u&=h(1-w)\left[e_M-\frac{\nabla F(Y)\cdot e_M}{\nabla F(Y)\cdot v_r}v_r\right],\\
        Y_v&=h w\left[e_N-\frac{\nabla F(Y)\cdot e_N}{\nabla F(Y)\cdot v_r}v_r\right].
\end{aligned}
\]
Hence $Y_u,Y_v=O(h)$.  The bracketed vectors are the projections of the two same-sign rail directions onto the tangent plane along the transverse direction $v_r$.  Uniform shape regularity of the tubular shell and of the original surface triangle implies that these two projected rail directions have a uniformly bounded condition number.  Thus the $2\times2$ metric for the ruled variables $(u,v)$ has eigenvalues comparable to $h^2$.

Finally, on each of the two triangles used to represent the generally
non-coplanar spatial quadrilateral, the ratio map
$(\xi_1,\xi_2)\mapsto(u,v)$ is a smooth map with Jacobian and inverse
Jacobian uniformly bounded.  This follows from the fact that the
denominators $\alpha+\gamma$ and $\beta+\delta$ in
\eqref{eq:ratio-map} are bounded below and from the shape regularity of
the cut configuration.  Composing the ruled metric with this uniformly
regular ratio map preserves the $h^2$ eigenvalue bounds.  This proves
\eqref{eq:metric-equivalence}.
\end{proof}

\begin{lemma}[Lift stability]
\label{lem:lift-stability}
For every $\phi_h\in\Sh^r$,
\begin{equation}
\label{eq:L2H1-stability}
        \norm{\phi_h^l}_{L^2(\Gamma)}\le C\norm{\phi_h}_{L^2(\Gamma_h)},
        \qquad
        \norm{\phi_h^l}_{H^1(\Gamma)}\le C\norm{\phi_h}_{H^1(\Gamma_h)}.
\end{equation}
The same estimates hold in the reverse direction for functions on $\Gamma$ pulled back to $\Gamma_h$.
\end{lemma}

\begin{proof}
This comes directly from the above bounds in \ref{lem:metric-bounds}.
\end{proof}

\subsection{High-order derivative estimates}
\label{sec:high-order}
This section proves the derivative estimates that control the broken
$H^m$ norm of the pullback from $\Gamma$ to $\Gamma_h$.  The estimates
are local; all constants are independent of the cut position.

\begin{lemma}[Integer inequality]
\label{lem:integer}
Let $q\ge2$ and let $n_1,\ldots,n_q$ be positive integers.  Then
\begin{equation}
\label{eq:integer}
        \sum_{j=1}^q \left\lfloor\frac{n_j-1}{2}\right\rfloor
        \le
        \left\lfloor\frac{\sum_{j=1}^q n_j}{2}\right\rfloor-1.
\end{equation}
Equivalently,
\begin{equation}
\label{eq:integer-scaled}
        \sum_{j=1}^q \left(\left\lfloor\frac{n_j}{2}\right\rfloor+1\right)
        \ge
        \left\lfloor\frac{\sum_{j=1}^q n_j}{2}\right\rfloor+1.
\end{equation}
\end{lemma}

\begin{proof}
For $q=2$ the statement is checked by the four parity cases.  For $q\ge3$,
\[
        \sum_{j=1}^q \left\lfloor\frac{n_j-1}{2}\right\rfloor
        \le \frac12\sum_{j=1}^q n_j-\frac q2
        \le \frac12\sum_{j=1}^q n_j-\frac32
        \le \left\lfloor\frac{\sum_j n_j}{2}\right\rfloor-1.
\]
The scaled form is the same inequality after rewriting the floors.
\end{proof}

\subsubsection{Triangular patches}
For the triangular cut, the following estimate is obtained by differentiating the vertex-projection constraint \eqref{eq:case1-implicit}.  It is the surface analogue of the two-dimensional curve estimate: the derivative of the signed distance along a chord is $O(h)$, because the chord endpoints lie on $\Gamma$ and the second derivative of $d$ is uniformly bounded.

\begin{theorem}[Derivative bounds for triangular cuts]
\label{thm:case1-derivatives}
Let $r_K:\wh K\to\Gamma\cap T$ be the triangular-patch lift, and let $\bs=(s_1,s_2)$ be any smooth local coordinate system on $\Gamma$ with uniformly bounded derivatives.  Write
\[
        r_K(\bxi)=g(\bs(\bxi)).
\]
For every multi-index $\balpha$ with $|\balpha|=n\ge1$,
\begin{equation}
\label{eq:case1-scaled-bound}
        |D_{\bxi}^{\balpha}\bs|\le C h^{\lfloor n/2\rfloor+1}.
\end{equation}
Equivalently, with derivatives taken with respect to physical coordinates on $K$,
\begin{equation}
\label{eq:case1-physical-bound}
        |D_X^{\balpha}\bs|\le C h^{-\lfloor (n-1)/2\rfloor}.
\end{equation}
\end{theorem}

\begin{proof}
We give the proof because the same mechanism is used later for the ruled quadrilateral lift.  Let
\[
        V(\bxi)=\frac{X(\bxi)-x_0}{h}.
\]
Since $Y=r_K(\bxi)$ lies on the segment $x_0X(\bxi)$, there is a scalar $k(\bxi)$ such that
\begin{equation}
\label{eq:tri-offset}
        g(\bs(\bxi))-x_0-k(\bxi)V(\bxi)=0.
\end{equation}
The scalar $k$ is of size $h$.  Indeed, the root $Y$ lies in a tetrahedron of diameter $O(h)$, while $V=O(1)$ and $V$ has a normal component bounded below.  More precisely, Taylor expansion of $d$ at $Y$ gives
\[
        d(x_0)=d(Y-kV)=-k\,\nn(Y)\cdot V+O(h^2),
\]
so $|\nn(Y)\cdot V|\ge c>0$ and $k\simeq h$.

We shall also use the following elementary geometric estimate.  Since the three vertices of the linear triangle $K$ lie on $\Gamma$ and $D^2d$ is bounded, the function $d(X(\bxi))$ satisfies
\begin{equation}
\label{eq:d-on-linear-triangle}
        \|d(X(\cdot))\|_{W^{1,\infty}(\wh K)}\le C h^2.
\end{equation}
This follows by applying the interpolation remainder estimate on the reference triangle to $d\circ X$; its second derivatives are $O(h^2)$ because $D_{\bxi}X=O(h)$.  In particular,
\begin{equation}
\label{eq:normal-gain-triangle}
        \left|\nn(Y)\cdot D_i V\right|
        =h^{-1}\left|\nn(Y)\cdot D_iX\right|
        \le C h,
        \qquad i=1,2,
\end{equation}
where the replacement of $\nn(X)$ by $\nn(Y)$ only changes the estimate by $O(h)$ because $|X-Y|=O(h^2)$ and $D_iX=O(h)$.

Set
\[
        A(\bxi)=\left(g_{s_1}(\bs),g_{s_2}(\bs),-V(\bxi)\right).
\]
The first two columns span $T_Y\Gamma$, and the third column has a normal component bounded away from zero.  Hence $A^{-1}$ is uniformly bounded.  Moreover, if
\[
        A\begin{pmatrix}a_1\\a_2\\a_3\end{pmatrix}=R,
\]
then the third component satisfies
\begin{equation}
\label{eq:Ainv-third-tri}
        |a_3|\le C|\nn(Y)\cdot R|+Ch|R|,
\end{equation}
because the tangential components of $V$ and the variation of the normal over one element are uniformly bounded, while the normal component of $V$ is bounded below.

We prove the stronger induction statement
\begin{equation}
\label{eq:tri-induction-claim}
        |D_{\bxi}^{\balpha}\bs|\le C h^{\lfloor n/2\rfloor+1},
        \qquad
        |D_{\bxi}^{\balpha}k|\le C h^{\lfloor (n+1)/2\rfloor+1},
        \qquad n=|\balpha|\ge1.
\end{equation}
For $n=1$, differentiating \eqref{eq:tri-offset} gives
\[
        A\begin{pmatrix}D_i s_1\\D_i s_2\\D_i k\end{pmatrix}=kD_iV.
\]
The right-hand side has size $O(h)$.  Its normal component is $O(h^2)$ by \eqref{eq:normal-gain-triangle}.  The boundedness of $A^{-1}$ and \eqref{eq:Ainv-third-tri} therefore give
\[
        D_i\bs=O(h),
        \qquad
        D_i k=O(h^2),
\]
which is the base case.

Assume now that \eqref{eq:tri-induction-claim} holds for all orders smaller than $n$, and let $|\balpha|=n\ge2$.  Apply $D_{\bxi}^{\balpha}$ to \eqref{eq:tri-offset}.  Since $V$ is affine in $\bxi$, all derivatives $D_{\bxi}^{\bgamma}V$ with $|\bgamma|\ge2$ vanish.  Isolating the highest-order derivatives gives
\begin{equation}
\label{eq:tri-linear-system}
        A\begin{pmatrix}D^{\balpha}s_1\\D^{\balpha}s_2\\D^{\balpha}k\end{pmatrix}=R_{\balpha},
\end{equation}
where $R_{\balpha}$ is a finite sum of the following lower-order terms:
\begin{equation}
\label{eq:tri-Ralpha}
\begin{aligned}
        R_{\balpha}=&-
        \sum_{m\ge2}\sum_{\bbeta_1+\cdots+\bbeta_m=\balpha}
        C_{\bbeta_1\cdots\bbeta_m}
        D_{\bs}^{m}g(\bs)
        \big[D^{\bbeta_1}\bs,\ldots,D^{\bbeta_m}\bs\big]\\
        &+\sum_{i=1}^2 \alpha_i\,D^{\balpha-e_i}k\,D_iV.
\end{aligned}
\end{equation}
In the first sum all $\bbeta_j$ are nonzero and have order strictly less than $n$.

The size estimate follows directly from the induction hypothesis.  For the $g$-terms, Lemma \ref{lem:integer} gives
\[
        \prod_{j=1}^m |D^{\bbeta_j}\bs|
        \le C h^{\sum_j(\lfloor |\bbeta_j|/2\rfloor+1)}
        \le C h^{\lfloor n/2\rfloor+1}.
\]
For the terms involving $V$, $D_iV=O(1)$ and
\[
        |D^{\balpha-e_i}k|
        \le C h^{\lfloor n/2\rfloor+1}.
\]
Thus
\begin{equation}
\label{eq:tri-R-size}
        |R_{\balpha}|\le C h^{\lfloor n/2\rfloor+1}.
\end{equation}
We also need a normal estimate.  For the $g$-terms, the same induction hypothesis and the parity estimate imply
\[
        \left|\nn(Y)\cdot D_{\bs}^{m}g(\bs)
        \big[D^{\bbeta_1}\bs,\ldots,D^{\bbeta_m}\bs\big]\right|
        \le C h^{\lfloor (n+1)/2\rfloor+1}.
\]
For the $V$-terms, \eqref{eq:normal-gain-triangle} gives one additional factor $h$:
\[
        \left|\nn(Y)\cdot D^{\balpha-e_i}k\,D_iV\right|
        \le C h^{\lfloor n/2\rfloor+2}
        \le C h^{\lfloor (n+1)/2\rfloor+1}.
\]
Consequently,
\begin{equation}
\label{eq:tri-R-normal}
        |\nn(Y)\cdot R_{\balpha}|
        \le C h^{\lfloor (n+1)/2\rfloor+1}.
\end{equation}
Combining \eqref{eq:tri-linear-system}, \eqref{eq:tri-R-size}, and \eqref{eq:tri-R-normal} with the mapping properties of $A^{-1}$ proves \eqref{eq:tri-induction-claim}.  In particular, \eqref{eq:case1-scaled-bound} follows.

Finally, $X_K:\wh K\to K$ is affine with $D_{\bxi}X_K=O(h)$ and inverse derivative $D_X\bxi=O(h^{-1})$.  Therefore $D_X^{\balpha}$ scales like $h^{-n}D_{\bxi}^{\balpha}$, and \eqref{eq:case1-physical-bound} follows from \eqref{eq:case1-scaled-bound}.
\end{proof}

\subsubsection{Quadrilateral patches: high-order estimate for the ruled lift}
\label{subsec:case2-proof}
We now prove the corresponding estimate for the four-edge cut.  This is the main new technical result.

Fix one of the two linear triangles $K_h\subset\Gamma_h\cap T$ in a four-edge cut, and let
\[
        X=X(\bxi):\wh K\to K_h
\]
be its affine map.  Let $u(\bxi),v(\bxi)$ be the ratio variables \eqref{eq:ratio-map}, and define $M,N$ as in \eqref{eq:rails}.  Let
\[
        q(\bxi)=\frac{N(v(\bxi))-M(u(\bxi))}{h}.
\]
The exact lift can be written in the form
\begin{equation}
\label{eq:case2-offset}
        Y(\bxi)=X(\bxi)+\ell(\bxi)q(\bxi),
\end{equation}
where $\ell=O(h^2)$.  Indeed, $X$ lies on the segment $M(u)N(v)$ by \eqref{eq:X-MN}, and $Y$ is the exact intersection of the same segment with $\Gamma$; the signed-distance error of the linear triangle is $O(h^2)$ since its vertices lie on $\Gamma$ and $D^2d$ is bounded.  Transversality of the segment then gives $\ell=O(h^2)$.

The balanced tubular offset condition gives the following normal-gain property for the ruled directions.

\begin{lemma}[Normal gain for the ruled direction]
\label{lem:normal-gain-q}
Let $q$ be defined above.  For every multi-index $\balpha$ with $m=|\balpha|\ge1$, the tubular offset shell mesh satisfies
\begin{equation}
\label{eq:q-derivative-bound}
        |D_{\bxi}^{\balpha}q|\le C h^{m-1},
        \qquad
        |\nn(Y)\cdot D_{\bxi}^{\balpha}q|\le C h^{m}.
\end{equation}
In particular, the first-order estimates are $|D_{\bxi}q|\le C$ and $|\nn(Y)\cdot D_{\bxi}q|\le Ch$.  Moreover, $|\nn(Y)\cdot q|\ge c>0$.
\end{lemma}

\begin{proof}
The ratio variables $u,v$ are rational functions of the tetrahedral barycentric coordinates.  On each of the two split triangles the denominators $\alpha+\gamma$ and $\beta+\delta$ in \eqref{eq:ratio-map} are bounded below by a positive constant.  The tubular offset shell construction gives more: these denominators are constants plus $O(h)$ affine perturbations on each split triangle.  Indeed, all same-sign vertices lie on the same signed-distance layer up to $O(h^2)$, and therefore the zero parameters on the four sign-changing edges differ from fixed constants only by $O(h)$.  Since the barycentric coordinates are affine in $\bxi$, repeated differentiation of the quotients gives
\begin{equation}
\label{eq:uv-higher-derivatives}
        |D_{\bxi}^{\balpha}u|+|D_{\bxi}^{\balpha}v|\le C h^{m-1},
        \qquad |\balpha|=m\ge1.
\end{equation}
We now prove (8.16).  Put
\[
        D_u=\alpha+\gamma,\qquad D_v=\beta+\delta,
        \qquad u=\gamma D_u^{-1},\qquad v=\delta D_v^{-1}.
\]
We first show that, on each split triangle,
\[
        D_u=c_u+\varepsilon_u,\qquad D_v=c_v+\varepsilon_v,
\]
where \(c_u,c_v>0\) are constants independent of \(h\), and
\[
        \|\varepsilon_u\|_{W^{1,\infty}(\widehat K)}
        +\|\varepsilon_v\|_{W^{1,\infty}(\widehat K)}
        \le Ch.
\]
Indeed, let \(A^+\) and \(A^-\) be the two endpoints of a sign-changing
edge, with
\[
        d(A^+)=\rho_+h+O(h^2),\qquad
        d(A^-)=-\rho_-h+O(h^2),
\]
where \(\rho_\pm>0\) are fixed.  If
\(X(t)=(1-t)A^+ + tA^-\), then
\[
        d(X(t))=(1-t)d(A^+)+t d(A^-)+R(t),
        \qquad |R(t)|+|R'(t)|\le Ch^2,
\]
because \(d\in C^2\) and \(|A^+-A^-|=O(h)\).  Hence the zero parameter
\(t_*\) satisfies
\[
        t_*=\frac{\rho_+}{\rho_++\rho_-}+O(h)
\]
when the edge is oriented from \(A^+\) to \(A^-\), and the analogous formula
with \(\rho_-\) holds for the opposite orientation.  Consequently, on the
triangle \(P_{01}P_{12}P_{30}\), the nodal values of \(D_u\) are all
\[
        c_u:=\frac{\rho_-}{\rho_++\rho_-}
\]
up to \(O(h)\), while the nodal values of \(D_v\) are all
\[
        c_v:=\frac{\rho_+}{\rho_++\rho_-}
\]
up to \(O(h)\).  The same statement holds on the second split triangle
\(P_{12}P_{23}P_{30}\).  Since the tetrahedral barycentric coordinates are
affine functions of the reference coordinates \(\xi\), the functions
\(D_u\) and \(D_v\) are affine on each split triangle.  Therefore their
deviations from \(c_u\) and \(c_v\) have \(W^{1,\infty}\)-norm \(O(h)\).
In particular,
\[
        D_u,D_v\ge c>0,\qquad
        |\nabla_\xi D_u|+|\nabla_\xi D_v|\le Ch,
        \qquad
        D_\xi^\mu D_u=D_\xi^\mu D_v=0\quad (|\mu|\ge2).
\]

It follows that, for every multi-index \(\mu\) with \(|\mu|=k\ge1\),
\[
        |D_\xi^\mu(D_u^{-1})|
        +|D_\xi^\mu(D_v^{-1})|
        \le C h^k.
\]
This is a direct consequence of the Faà di Bruno formula: since \(D_u\)
and \(D_v\) are affine, every nonzero term in \(D_\xi^\mu(D_u^{-1})\) or
\(D_\xi^\mu(D_v^{-1})\) contains exactly \(k\) first derivatives of the
denominator, each of which is \(O(h)\), while the denominator itself is
bounded away from zero.

Finally, since \(\gamma\) and \(\delta\) are affine functions of \(\xi\),
\[
        D_\xi^\mu(\gamma D_u^{-1})
        =
        \gamma D_\xi^\mu(D_u^{-1})
        +\sum_{i=1}^2 \mu_i(\partial_{\xi_i}\gamma)
          D_\xi^{\mu-e_i}(D_u^{-1}),
\]
and all higher derivatives of \(\gamma\) vanish.  Therefore, for
\(|\mu|=m\ge1\),
\[
        |D_\xi^\mu u|
        \le Ch^m+Ch^{m-1}
        \le Ch^{m-1}.
\]
The same argument applied to \(v=\delta D_v^{-1}\) gives
\[
        |D_\xi^\mu v|\le Ch^{m-1}.
\]
This proves (8.16).
Using
\[
        D_{\bxi}^{\balpha}q
        =\frac{x_3-x_1}{h}D_{\bxi}^{\balpha}v
         -\frac{x_2-x_0}{h}D_{\bxi}^{\balpha}u,
        \qquad |\balpha|\ge1,
\]
and the fact that all tetrahedral edges have length $O(h)$, we obtain the first estimate in \eqref{eq:q-derivative-bound}.

It remains to estimate the normal components of the rail directions.  Since $x_0$ and $x_2$ have the same sign and the mesh is balanced,
\[
        d(x_2)-d(x_0)=O(h^2).
\]
Taylor expansion of $d$ along $x_0x_2$ gives
\[
        \nabla d(x_0)\cdot (x_2-x_0)=O(h^2).
\]
The normal varies by $O(h)$ inside $T$, so
\[
        \nn(Y)\cdot\frac{x_2-x_0}{h}=O(h).
\]
The same argument applies to $(x_3-x_1)/h$.  Combining this $O(h)$ normal factor with \eqref{eq:uv-higher-derivatives} gives the second estimate in \eqref{eq:q-derivative-bound}.  Finally, $M$ and $N$ lie on opposite sides of $\Gamma$ and have signed distances comparable to $h$; hence $q=(N-M)/h$ has a normal component bounded away from zero.
\end{proof}

\begin{theorem}[Derivative bounds for four-edge cuts]
\label{thm:case2-derivatives}
Let $K_h$ be either triangle in the split of a quadrilateral cut, and let $Y(\bxi)=\Phi_T(X(\bxi))$.  Let $g(\bs)$ be a smooth local coordinate chart of $\Gamma$ such that $Y(\bxi)=g(\bs(\bxi))$.  Under Definition \ref{def:admissible} and the tubular offset shell construction, for every multi-index $\balpha$ with $|\balpha|=n\ge1$,
\begin{equation}
\label{eq:case2-scaled-bound}
        |D_{\bxi}^{\balpha}\bs|\le C h^{\lfloor n/2\rfloor+1}.
\end{equation}
Equivalently, in physical coordinates on $K_h$,
\begin{equation}
\label{eq:case2-physical-bound}
        |D_X^{\balpha}\bs|\le C h^{-\lfloor (n-1)/2\rfloor}.
\end{equation}
\end{theorem}

\begin{proof}
We use the offset equation \eqref{eq:case2-offset}.  Write it as
\begin{equation}
\label{eq:offset-system}
        \Psi(\bs,\ell,\bxi):=g(\bs)-X(\bxi)-\ell q(\bxi)=0.
\end{equation}
Let
\[
        A(\bxi)=\left(g_{s_1}(\bs),g_{s_2}(\bs),-q(\bxi)\right).
\]
Since $q$ has a normal component bounded away from zero, while $g_{s_1},g_{s_2}$ span the tangent plane, $A^{-1}$ is uniformly bounded.  More precisely, if $R\in\R^3$, then the third component of $A^{-1}R$ is bounded by $C|\nn(Y)\cdot R|+Ch|R|$, and the first two components are bounded by $C|R|$.

We prove by induction on $n=|\balpha|$ the stronger pair of estimates
\begin{equation}
\label{eq:induction-claim}
        |D_{\bxi}^{\balpha}\bs|\le C h^{\lfloor n/2\rfloor+1},
        \qquad
        |D_{\bxi}^{\balpha}\ell|\le C h^{\lfloor (n+1)/2\rfloor+1}.
\end{equation}
For $n=1$, differentiating \eqref{eq:offset-system} gives
\begin{equation}
\label{eq:case2-first}
        A\begin{pmatrix}D_i s_1\\D_i s_2\\D_i\ell\end{pmatrix}
        =D_iX+\ell D_iq.
\end{equation}
Here $D_iX=O(h)$ and $\ell D_iq=O(h^2)$.  Also
\[
        \nn(Y)\cdot D_iX=D_i(d(X))+O(|Y-X|\,|D_iX|)=O(h^2),
\]
because $d(X)=O(h^2)$ on the whole linear triangle.  Lemma \ref{lem:normal-gain-q} gives $\nn(Y)\cdot\ell D_iq=O(h^3)$.  Hence $D_i\bs=O(h)$ and $D_i\ell=O(h^2)$, proving the base case.

Let $n\ge2$ and assume \eqref{eq:induction-claim} for all lower orders.  Apply $D_{\bxi}^{\balpha}$ to \eqref{eq:offset-system}.  Since $X$ is affine, $D_{\bxi}^{\balpha}X=0$.  Isolating the highest-order derivatives gives
\begin{equation}
\label{eq:case2-linear-system}
        A\begin{pmatrix}D^{\balpha}s_1\\D^{\balpha}s_2\\D^{\balpha}\ell\end{pmatrix}
        =R_{\balpha},
\end{equation}
where $R_{\balpha}$ is a finite sum of two types of lower-order terms:
\begin{equation}
\label{eq:R-alpha}
\begin{aligned}
        R_{\balpha}=&-
        \sum_{m\ge2}\sum_{\bbeta_1+\cdots+\bbeta_m=\balpha}
        C_{\bbeta_1\cdots\bbeta_m}
        D_{\bs}^{m}g(\bs)
        \big[D^{\bbeta_1}\bs,\ldots,D^{\bbeta_m}\bs\big] \\
        &+\sum_{0\le \bbeta<\balpha}
        C_{\bbeta}\,D^{\bbeta}\ell\,D^{\balpha-\bbeta}q.
\end{aligned}
\end{equation}
In the first sum every $\bbeta_j$ is nonzero and has order strictly less than $n$; in the second sum $D^{\balpha-\bbeta}q$ is a nonzero derivative of $q$.

We first bound the size of $R_{\balpha}$.  For the terms involving $g$, the induction hypothesis and Lemma \ref{lem:integer} imply
\[
        \prod_{j=1}^m |D^{\bbeta_j}\bs|
        \le C h^{\sum_j(\lfloor |\bbeta_j|/2\rfloor+1)}
        \le C h^{\lfloor n/2\rfloor+1}.
\]
For the terms involving $q$, write $j=|\bbeta|$ and $m=|\balpha-\bbeta|=n-j\ge1$.  If $j=0$, then $D^{\bbeta}\ell=\ell=O(h^2)$ and Lemma \ref{lem:normal-gain-q} gives $D^{\balpha}q=O(h^{n-1})$, which is better than required.  If $j\ge1$, the induction hypothesis and Lemma \ref{lem:normal-gain-q} give
\[
        |D^{\bbeta}\ell|\,|D^{\balpha-\bbeta}q|
        \le
        C h^{\lfloor(j+1)/2\rfloor+1}h^{m-1}
        \le C h^{\lfloor n/2\rfloor+1}.
\]
Therefore
\begin{equation}
\label{eq:R-size}
        |R_{\balpha}|\le C h^{\lfloor n/2\rfloor+1}.
\end{equation}

We also need the normal component of $R_{\balpha}$.  The terms involving $g$ satisfy the stronger estimate
\[
        \left|\nn(Y)\cdot D_{\bs}^{m}g(\bs)
        \big[D^{\bbeta_1}\bs,\ldots,D^{\bbeta_m}\bs\big]\right|
        \le C h^{\lfloor (n+1)/2\rfloor+1},
\]
again by the induction hypothesis and the parity inequality; for $m=2$ and $n=2$ this is exactly the curvature term of size $O(h^2)$.  For the $q$-terms, Lemma \ref{lem:normal-gain-q} gives the normal estimate for $D^{\balpha-\bbeta}q$.  The same product estimate therefore yields
\begin{equation}
\label{eq:R-normal}
        |\nn(Y)\cdot R_{\balpha}|
        \le C h^{\lfloor (n+1)/2\rfloor+1}.
\end{equation}
Combining \eqref{eq:case2-linear-system}, \eqref{eq:R-size}, and \eqref{eq:R-normal} with the mapping properties of $A^{-1}$ yields
\[
        |D^{\balpha}\bs|\le C h^{\lfloor n/2\rfloor+1},
        \qquad
        |D^{\balpha}\ell|\le C h^{\lfloor (n+1)/2\rfloor+1}.
\]
This closes the induction.  Finally, since $D_{\bxi}X=O(h)$ and the inverse affine map has derivatives $O(h^{-1})$, scaling from the reference triangle to physical coordinates gives \eqref{eq:case2-physical-bound}.
\end{proof}

\paragraph{Remark on the proof framework.}
The tubular offset shell construction is the point at which the proof uses more than non-reintersection of the surface with tetrahedra.  Without \eqref{eq:balanced-distance}, the same-sign rails in a four-edge cut may have an $O(1)$ normal component after division by $h$, and the normal-gain Lemma \ref{lem:normal-gain-q} need not hold.  The construction is still well-defined, and the $H^1$ stability of the lift remains valid, but the high-order $H^m$ estimate for the ruled lift is not guaranteed by the present argument.  Therefore the mesh should be generated by the balanced tubular shell construction, or by another construction with the same normal-gain and ratio-map estimates, when a rigorous high-order error estimate is desired.

\subsubsection{Pullback estimates in Sobolev norms}
Combining Theorems \ref{thm:case1-derivatives} and \ref{thm:case2-derivatives} gives a uniform estimate for all surface elements.

For an integer \(m\ge0\), we use the broken Sobolev norm
\begin{equation}
\label{eq:broken-Sobolev-norm}
        \norm{w}_{H^m(\Kh)}^2
        :=
        \sum_{K\in\Kh}\norm{w}_{H^m(K)}^2 .
\end{equation}
Here \(H^m(\Kh)\) denotes the corresponding broken space.  This
distinction is essential for \(m\ge2\): the global lift is continuous,
but derivatives of its elementwise restrictions need not agree across
surface-element edges.  Thus a smooth function on \(\Gamma\) generally
pulls back to a piecewise \(H^m\) function, not to an element of the
global space \(H^m(\Gamma_h)\).

\begin{theorem}[Broken Sobolev norm of the pullback]
\label{thm:pullback-Hm}
Let $m\ge1$ be an integer and assume $F\in C^{m+2}$ and the adjusted mesh is admissible.  Let $v\in H^m(\Gamma)$, and define its pullback to $\Gamma_h$ by
\[
        v^{-l}(X)=v(\Phi_h(X)),\qquad X\in\Gamma_h.
\]
Then
\begin{equation}
\label{eq:Hm-pullback}
        \norm{v^{-l}}_{H^m(\Kh)}
        \le C h^{-\lfloor (m-1)/2\rfloor}\norm{v}_{H^m(\Gamma)}.
\end{equation}
For $m=0$ and $m=1$, the pullback is globally conforming and this
estimate is consistent with the uniform norm equivalence of Lemma
\ref{lem:lift-stability}.
\end{theorem}

\begin{proof}
The proof is local.  In a smooth chart $g(\bs)$ of $\Gamma$, derivatives of $v\circ g(\bs(X))$ with respect to physical coordinates on $K\subset\Gamma_h$ are sums, by the multivariate chain rule, of terms of the form
\[
        D_{\bs}^q(v\circ g)
        \prod_{j=1}^q D_X^{\bbeta_j}\bs,
        \qquad
        \bbeta_1+\cdots+\bbeta_q=\balpha,
        \quad |\balpha|\le m.
\]
Theorems \ref{thm:case1-derivatives} and \ref{thm:case2-derivatives} give
\[
        |D_X^{\bbeta_j}\bs|
        \le C h^{-\lfloor(|\bbeta_j|-1)/2\rfloor}.
\]
If $q=1$, the factor is at most $h^{-\lfloor(|\balpha|-1)/2\rfloor}$.  If $q\ge2$, Lemma \ref{lem:integer} gives the same or a better power.  Applying these estimates elementwise, summing over $K\in\Kh$, and using the uniform metric equivalence proves \eqref{eq:Hm-pullback}.  No matching of derivatives across element edges is used.
\end{proof}

\subsection{Error estimate}
\label{sec:error-estimate}

Let $I_h$ denote the standard nodal interpolation operator on the surface
triangulation $\Kh$.  For $u\in H^{r+1}(\Gamma)$, define its pullback
$\wt u=u^{-l}$ on $\Gamma_h$, and set
\[
        v_h^l=(I_h\wt u)^l\in S_{h,r}^l .
\]
By Lemma \ref{lem:lift-stability}, the standard interpolation estimate on the
shape-regular surface triangulation $\Kh$, and Theorem \ref{thm:pullback-Hm},
\begin{equation}
\label{eq:interp-H1}
\begin{aligned}
        \norm{u-v_h^l}_{H^1(\Gamma)}
        &\le C\norm{\wt u-I_h\wt u}_{H^1(\Gamma_h)}\\
        &\le C h^r \norm{\wt u}_{H^{r+1}(\Kh)}\\
        &\le C h^{r-\lfloor r/2\rfloor}\norm{u}_{H^{r+1}(\Gamma)}\\
        &= C h^{\lceil r/2\rceil}\norm{u}_{H^{r+1}(\Gamma)} .
\end{aligned}
\end{equation}
Together with Cea's lemma \eqref{eq:cea}, this proves the following energy-norm
estimate.

\begin{theorem}[Energy error]
\label{thm:H1-error}
Let $u$ solve \eqref{eq:model-pde}, and let $u_h^l\in S_{h,r}^l$ solve
\eqref{eq:discrete}.  Assume $u\in H^{r+1}(\Gamma)$, $F\in C^{r+3}$, and
the adjusted mesh is admissible.  Then
\begin{equation}
\label{eq:H1-error}
        \norm{u-u_h^l}_{H^1(\Gamma)}
        \le C h^{\lceil r/2\rceil}\norm{u}_{H^{r+1}(\Gamma)} .
\end{equation}
Consequently, since $\norm{w}_{L^2(\Gamma)}\le \norm{w}_{H^1(\Gamma)}$,
one also has the direct Cea-type $L^2$ estimate
\begin{equation}
\label{eq:L2-error-cea}
        \norm{u-u_h^l}_{L^2(\Gamma)}
        \le C h^{\lceil r/2\rceil}\norm{u}_{H^{r+1}(\Gamma)} .
\end{equation}
\end{theorem}

The estimate \eqref{eq:L2-error-cea} is the strongest $L^2$ estimate that follows
directly from Cea's lemma alone.  The interpolation function $v_h^l$ itself satisfies
the sharper approximation estimate
\begin{equation}
\label{eq:L2-best-approx}
\begin{aligned}
        \norm{u-v_h^l}_{L^2(\Gamma)}
        &\le C\norm{\wt u-I_h\wt u}_{L^2(\Gamma_h)}\\
        &\le C h^{r+1}\norm{\wt u}_{H^{r+1}(\Kh)}\\
        &\le C h^{r+1-\lfloor r/2\rfloor}
             \norm{u}_{H^{r+1}(\Gamma)}\\
        &= C h^{\lceil r/2\rceil+1}
             \norm{u}_{H^{r+1}(\Gamma)} .
\end{aligned}
\end{equation}

To obtain the additional power of $h$ for the Galerkin error, one uses the usual
Aubin--Nitsche duality argument.  Let $e=u-u_h^l$, and let $z\in H^1(\Gamma)$ solve
\[
        -\Delta_\Gamma z+z=e
        \qquad\text{on }\Gamma .
\]
Assume the elliptic regularity estimate
\begin{equation}
\label{eq:dual-regularity}
        \norm{z}_{H^2(\Gamma)}
        \le C\norm{e}_{L^2(\Gamma)} .
\end{equation}
By Galerkin orthogonality, for any $z_h^l\in S_{h,r}^l$,
\[
        \norm{e}_{L^2(\Gamma)}^2
        =
        a(e,z)
        =
        a(e,z-z_h^l).
\]
Choosing $z_h^l=(I_h z^{-l})^l$ and using the $H^1$ interpolation estimate with
only $H^2$ regularity gives
\[
        \norm{z-z_h^l}_{H^1(\Gamma)}
        \le C h \norm{z}_{H^2(\Gamma)} .
\]
Therefore,
\[
\begin{aligned}
        \norm{e}_{L^2(\Gamma)}^2
        &\le C\norm{e}_{H^1(\Gamma)}
              \norm{z-z_h^l}_{H^1(\Gamma)}\\
        &\le C h\norm{e}_{H^1(\Gamma)}
              \norm{z}_{H^2(\Gamma)}\\
        &\le C h\norm{e}_{H^1(\Gamma)}
              \norm{e}_{L^2(\Gamma)} .
\end{aligned}
\]
Combining this with \eqref{eq:H1-error} yields the sharper $L^2$ estimate
\begin{equation}
\label{eq:L2-error}
        \norm{u-u_h^l}_{L^2(\Gamma)}
        \le C h^{\lceil r/2\rceil+1}\norm{u}_{H^{r+1}(\Gamma)} .
\end{equation}

The estimates above display the half-order-type derivative loss caused by the
general lifting map.  If a stronger lift with uniformly bounded higher derivatives is
used, the same proof recovers the usual optimal surface finite element rates.

\section{Numerical experiments}
\label{sec:numerical-experiments}

We test the lifted surface finite element method on an ellipsoid and a torus.
Both experiments use manufactured smooth solutions and the same tubular-shell
construction, local lifts, polynomial degrees, and quadrature orders.  The
purpose is to verify the estimates of Section~4.3 on geometries with different
curvature and topology.  For a finite element degree \(p\), the theoretical
reference orders are
\[
    \|u-u_h^l\|_{H^1(\Gamma)}
    =O\!\left(h^{\lceil p/2\rceil}\right),
    \qquad
    \|u-u_h^l\|_{L^2(\Gamma)}
    =O\!\left(h^{\lceil p/2\rceil+1}\right).
\]
Thus the \(H^1/L^2\) reference orders are \(2/3\) for \(P_3\) and \(3/4\)
for \(P_5\).

\subsection{Ellipsoid test problem}
\label{subsec:ellipsoid-test}

As a first test, we consider the ellipsoid
\begin{equation}
    \Gamma
    =
    \left\{(x,y,z)\in\mathbb{R}^{3}:F(x,y,z)=0\right\},
    \qquad
    F(x,y,z)
    =
    \frac{x^{2}}{a^{2}}
    +\frac{y^{2}}{b^{2}}
    +\frac{z^{2}}{c^{2}}-1,
    \label{eq:ellipsoid-levelset}
\end{equation}
with semi-axes \((a,b,c)=(2,1.5,1)\).  We solve
\[
    -\Delta_{\Gamma}u+u=f
    \qquad\hbox{on }\Gamma
\]
by the method of manufactured solutions and prescribe
\begin{equation}
    u(x,y,z)=\exp(k_1x+k_2y+k_3z),
    \qquad
    k=
    \left(\frac{0.55}{a},-\frac{0.35}{b},\frac{0.75}{c}\right)^{T}.
    \label{eq:ellipsoid-manufactured-solution}
\end{equation}
Let \(g=\nabla F\), \(H_F=D^2F\), and \(n_\Gamma=g/|g|\).  The divergence of
the unit normal is available explicitly:
\[
    \operatorname{div}n_\Gamma
    =
    \frac{\operatorname{tr}H_F}{|g|}
    -\frac{g^T H_F g}{|g|^3}.
\]
Using
\[
    \Delta_\Gamma U
    =
    (I-n_\Gamma n_\Gamma^T):D^2U
    -(n_\Gamma\cdot\nabla U)\operatorname{div}n_\Gamma,
\]
the right-hand side is
\begin{equation}
    f
    =
    \left[
       1-|k|^2+(k\cdot n_\Gamma)^2
       +(k\cdot n_\Gamma)\operatorname{div}n_\Gamma
    \right]u.
    \label{eq:ellipsoid-rhs}
\end{equation}

A shape-regular and quasi-uniform auxiliary triangulation is obtained by
subdividing the eight faces of a regular octahedron with frequencies
\(N=4,6,8,10\), projecting the new vertices radially to the unit sphere, and
applying the smooth scaling map
\[
    (\xi_1,\xi_2,\xi_3)
    \longmapsto
    (a\xi_1,b\xi_2,c\xi_3).
\]
For every surface vertex \(p_i\), the two offset vertices are
\[
    p_i^\pm=p_i\pm0.35h_S n_\Gamma(p_i).
\]
Each resulting triangular prism is split into three tetrahedra, after which
the three-edge and four-edge local lifts are used to assemble the method on
the exact ellipsoid.

Continuous Lagrange elements of degrees \(p=3\) and \(p=5\) are tested.  A
Duffy-transformed triangular Gauss rule of order \(16\) is used for assembly,
and a rule of order \(22\) is used for the error norms.  A least-squares fit
over all four mesh levels gives
\[
    \begin{array}{c|cc}
        p & H^1\hbox{-error} & L^2\hbox{-error} \\
        \hline
        3 & 2.46 & 3.71 \\
        5 & 3.84 & 5.05
    \end{array}.
\]
All four levels contain no tiny surface triangles, and the same-sign offset
imbalance remains at roundoff level.  The measured rates meet or exceed the
theoretical reference orders.

\begin{figure}[t]
    \centering
    \includegraphics[width=\linewidth]{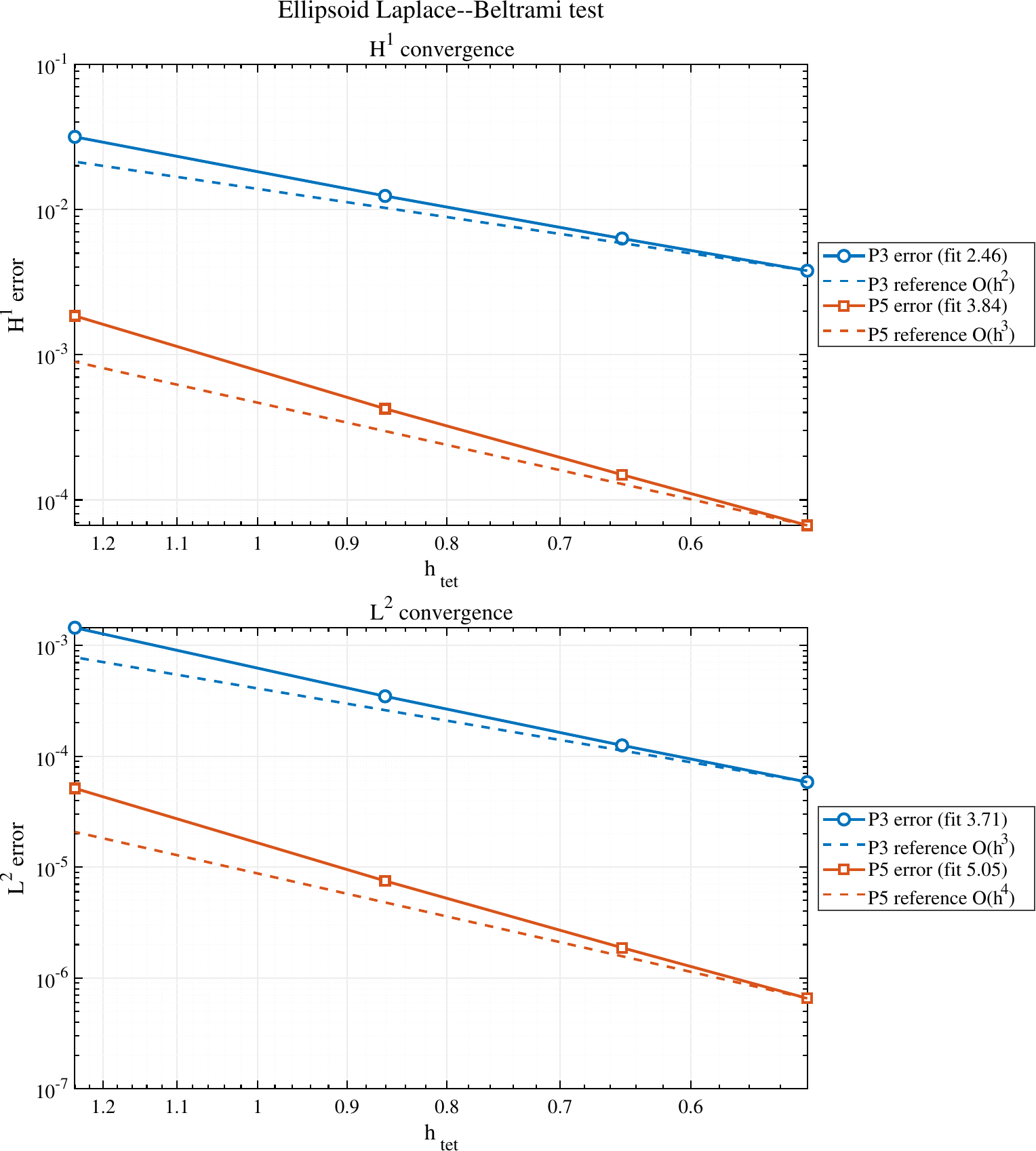}
    \caption{Log-log convergence histories for the ellipsoid problem.  The
    \(H^1\)-error is shown in the upper panel and the \(L^2\)-error in the
    lower panel.  Dashed lines indicate the theoretical reference orders;
    the fitted orders in the legends are least-squares slopes over all four
    mesh levels.}
    \label{fig:ellipsoid-p3p5-stacked-v2}
\end{figure}

\subsection{Torus test problem}
\label{subsec:torus-test}

As a second test, we consider the torus
\begin{equation}
    \Gamma
    =
    \left\{(x,y,z)\in\mathbb{R}^{3}:F(x,y,z)=0\right\},
    \qquad
    F
    =
    (x^{2}+y^{2}+z^{2}+R^{2}-a^{2})^{2}
    -4R^{2}(x^{2}+y^{2}),
    \label{eq:torus-levelset}
\end{equation}
with major radius \(R=2\) and minor radius \(a=0.7\).  Equivalently,
\[
    X(\theta,\varphi)
    =
    \bigl((R+a\cos\varphi)\cos\theta,
          (R+a\cos\varphi)\sin\theta,
          a\sin\varphi\bigr),
    \qquad
    (\theta,\varphi)\in[0,2\pi)^2.
\]
We solve
\[
    -\Delta_{\Gamma}u+u=f
    \qquad\hbox{on }\Gamma
\]
by the method of manufactured solutions and prescribe
\begin{equation}
    u(x,y,z)=\exp(k_1x+k_2y+k_3z),
    \qquad
    k=(0.15,-0.125,0.5)^T.
    \label{eq:torus-manufactured-solution}
\end{equation}
The outward unit normal and its divergence are
\[
    n_\Gamma
    =
    (\cos\varphi\cos\theta,
     \cos\varphi\sin\theta,
     \sin\varphi)^T,
    \qquad
    \operatorname{div}n_\Gamma
    =
    \frac{1}{a}
    +\frac{\cos\varphi}{R+a\cos\varphi}.
\]
The same surface-Laplacian identity used above gives
\begin{equation}
    f
    =
    \left[
       1-|k|^2+(k\cdot n_\Gamma)^2
       +(k\cdot n_\Gamma)\operatorname{div}n_\Gamma
    \right]u.
    \label{eq:torus-rhs}
\end{equation}

The torus has a regular doubly periodic parametrization with metric
\[
    g_{\theta\theta}=(R+a\cos\varphi)^2,
    \qquad
    g_{\varphi\varphi}=a^2,
    \qquad
    g_{\theta\varphi}=0.
\]
Since \(R-a\le R+a\cos\varphi\le R+a\), a shape-regular and quasi-uniform
surface triangulation is obtained from uniform periodic partitions with
\(N_\varphi=6,8,10,12\) and
\[
    N_\theta
    =
    \max\!\{12,\lceil
    \frac{R}{a}N_\varphi\rceil\},
\]
followed by a fixed diagonal subdivision of every parameter rectangle.  For
each surface vertex \(p_i\), we set
\[
    p_i^\pm=p_i\pm0.35h_S n_\Gamma(p_i),
\]
split each triangular prism into three tetrahedra, and apply the same
three-edge and four-edge local lifts as in the ellipsoid test.

Continuous Lagrange elements of degrees \(p=3\) and \(p=5\) are tested.  The
assembly and error quadrature orders are again \(16\) and \(22\), respectively.
A least-squares fit over all four mesh levels gives
\[
    \begin{array}{c|cc}
        p & H^1\hbox{-error} & L^2\hbox{-error} \\
        \hline
        3 & 2.63 & 3.93 \\
        5 & 3.96 & 5.24
    \end{array}.
\]
No tiny surface triangles are detected on any level, and the same-sign offset
imbalance is at machine precision.  The measured rates again meet or exceed
the theoretical reference orders, now on a non-convex surface of genus one.

\begin{figure}[t]
    \centering
    \includegraphics[width=\linewidth]{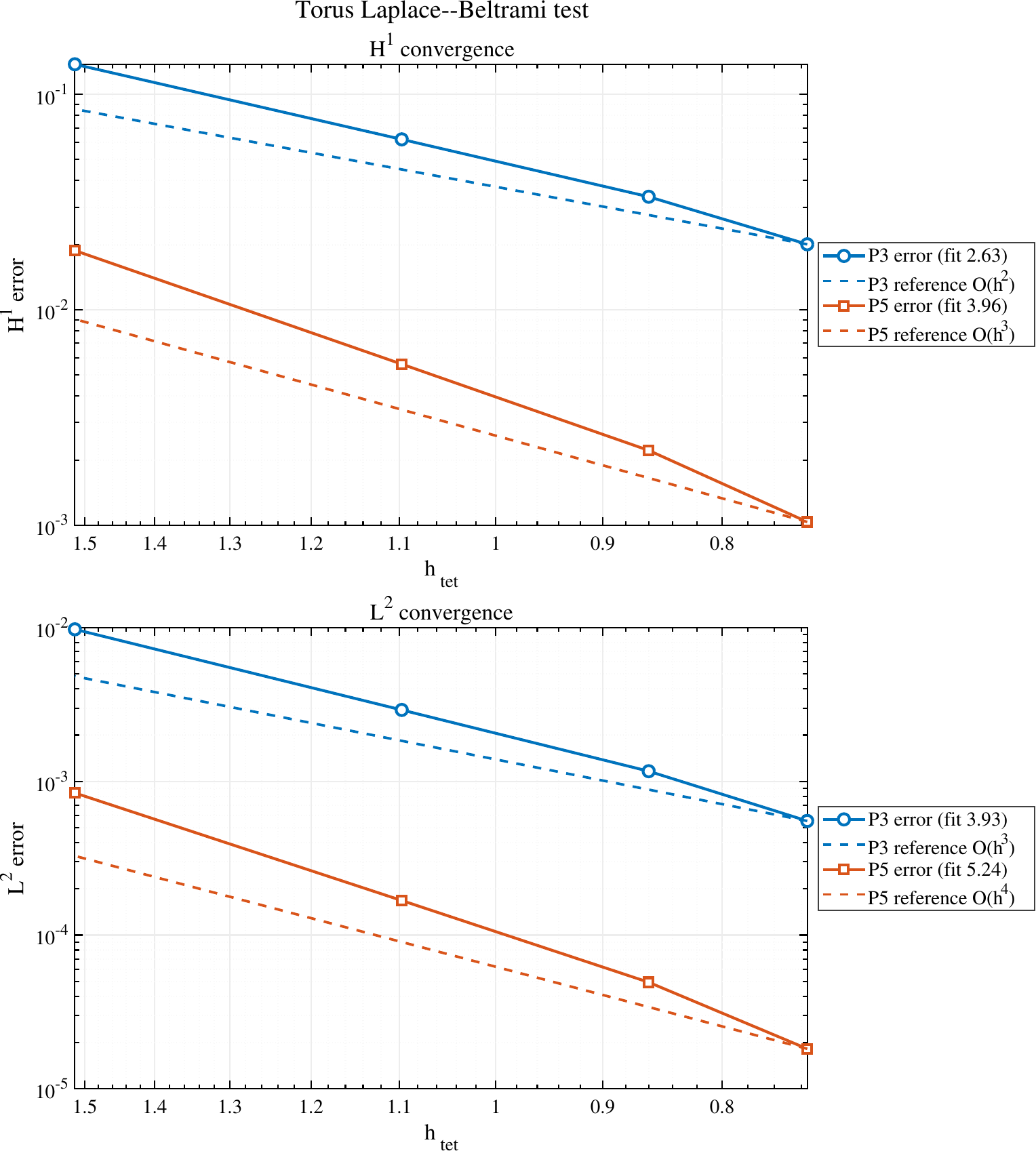}
    \caption{Log-log convergence histories for the torus problem.  The
    \(H^1\)-error is shown in the upper panel and the \(L^2\)-error in the
    lower panel.  Dashed lines indicate the theoretical reference orders;
    the fitted orders in the legends are least-squares slopes over all four
    mesh levels.}
    \label{fig:torus-p3p5-stacked-v2}
\end{figure}

\section{Conclusion}
\label{sec:conclusion}
We presented a high-order surface finite element framework for elliptic equations on implicitly defined surfaces.  The method combines a tubular offset tetrahedral shell mesh, exact local surface parameterizations, and a conforming lifted finite element space.  The shell construction enforces the balanced normal condition while preserving shape regularity through comparable tangential and normal length scales.  The triangular cut is handled by the vertex-projection map.  The four-edge cut is handled by a ruled two-rail map, which provides a single lift for the two triangles representing the generally non-coplanar spatial quadrilateral.  We proved that the local lifts agree pointwise on common mesh faces and assemble into a global homeomorphism.  The main analysis proves high-order derivative estimates for triangular vertex-projection lifts and for four-edge ruled lifts under the stated admissibility assumptions.  These estimates yield explicit broken-Sobolev pullback bounds and a Cea-type energy error estimate for the lifted method.

\section*{Acknowledgments}
This work was supported by National Natural Science \\
Foundation of China (NSFC) 92370125.

\bibliographystyle{siamplain}
\bibliography{references}

\end{document}